\documentclass[11pt,reqno]{amsart}

\usepackage{amssymb}

\usepackage{amsfonts}

\usepackage{amsmath}

\usepackage[all]{xy}

\usepackage{color}

\newtheorem{theorem}{Theorem}[section]

\newtheorem{corollary}[theorem]{Corollary}

\newtheorem{lemma}[theorem]{Lemma}

\newtheorem{proposition}[theorem]{Proposition}

\newtheorem{Definition}[theorem]{Definition}

\newtheorem{Example}[theorem]{Example}

\newtheorem{Remark}[theorem]{Remark}

\newenvironment{remark}{\begin{Remark}\begin{em}}{\end{em}\end{Remark}}

\DeclareMathOperator{\tr}{tr}

\address{Jinmi Hwang \\ Department of Mathematics, Chungbuk National University, Cheongju 28644, Korea}
\email{jinmi0401@chungbuk.ac.kr}

\address{Sejong Kim \\ Department of Mathematics, Chungbuk National University, Cheongju 28644, Korea}
\email{skim@chungbuk.ac.kr}

\begin{document}

\author[Hwang and Kim]{Jinmi Hwang and Sejong Kim}

\title[Bounds for the Wasserstein mean]{Bounds for the Wasserstein mean with applications to the Lie-Trotter mean}

\date{}
\maketitle

\begin{abstract}
As the least squares mean for the Riemannian trace metric on the cone of positive definite matrices, the Riemannian mean with its computational and theoretical approaches has been widely studied. Recently the new metric and the least squares mean on the cone of positive definite matrices, which are called the Wasserstein metric and the Wasserstein mean, respectively, have been introduced. In this paper, we explore some properties of Wasserstein mean such as determinantal inequality and find bounds for the Wasserstein mean. Using bounds for the Wasserstein mean, we verify that the Wasserstein mean is the multivariate Lie-Trotter mean.

\vspace{5mm}

%\noindent {\bf Mathematics Subject Classification} (2010):

\end{abstract}

\section{Introduction}
The open convex cone $\mathbb{P}_{m}$ of all $ m \times m $ positive definite Hermitian matrices with the inner product $\langle X , Y \rangle_{A} = \tr( A^{-1} X A^{-1} Y)$ on the tangent space $T_{A} (\mathbb{P}_{m})$ at each point $A \in \mathbb{P}_{m}$ gives us a Riemannian structure. Indeed, $\mathbb{P}_{m}$ is a Cartan-Hadamard manifold, a simply connected complete Riemannain manifold with non-positive sectional curvature, and also a Hadamard space. The Riemannian trace metric between $A$ and $B$ is given by $\delta(A,B) = \| \log A^{-1/2} B A^{-1/2} \|_{2}$, where $\Vert \cdot \Vert_{2}$ denotes the Frobenius norm. The natural and canonical mean on a Hadamard space is the least squares mean, called the Cartan mean or Riemannian mean.
For a positive probability vector $\omega = (w_{1}, \dots, w_{n})$, the Riemannian mean of $A_{1}, \dots, A_{n} \in \mathbb{P}_{m}$ is defined as
\begin{equation}
\Lambda(\omega; A_{1}, \dots, A_{n}) = \underset{X \in \mathbb{P}_{m}} {\arg \min} \sum_{j=1}^{n} w_{j} \delta^{2}(X, A_{j}).
\end{equation}
The Riemannian mean with its computational and theoretical approaches has been widely studied. One of the important properties of the Riemannian mean is the arithmetic-geometric-harmonic mean inequalities:
\begin{equation} \label{E:AGH inequalities}
\left[ \sum_{j=1}^{n} w_{j} A_{j}^{-1} \right]^{-1} \leq \Lambda(\omega; A_{1}, \dots, A_{n}) \leq \sum_{j=1}^{n} w_{j} A_{j}.
\end{equation}
Using \eqref{E:AGH inequalities} it has been verified in \cite{HK} that the Riemannian mean is the multivariate Lie-Trotter mean: for any differentiable curves $\gamma_{1}, \dots, \gamma_{n}$ on $\mathbb{P}_{m}$ with $\gamma_{i}(0) = I$ for all $i$,
\begin{displaymath}
\lim_{s \to 0} \Lambda (\omega; \gamma_{1}(s), \gamma_{2}(s), \ldots, \gamma_{n}(s))^{1/s} = \exp \left[ \sum_{i=1}^{n} w_{i} \gamma_{i}'(0) \right].
\end{displaymath}

Bhatia, Jain, and Lim \cite{BJL} have recently introduced a new metric and the least squares mean on $\mathbb{P}_{m}$, called the Wasserstein metric and the Wasserstein mean, respectively. For given $A, B \in \mathbb{P}$, the \emph{Wasserstein metric} $d(A, B)$ is given by
\begin{equation}
d(A, B) = \left[ \tr \left( \frac{A + B}{2} \right) - \tr (A^{1/2} B A^{1/2})^{1/2} \right]^{1/2}.
\end{equation}
In quantum information theory, the value $\tr (A^{1/2} B A^{1/2})^{1/2}$ is known as the fidelity and the Wasserstein metric is known as the Bures distance of density matrices.
The geodesic passing from $A$ to $B$ is given by
\begin{displaymath}
\gamma(t) = (1-t)^{2} A + t^{2} B + t (1-t) [(A B)^{1/2} + (B A)^{1/2}] = A \diamond_{t} B
\end{displaymath}
for $t \in [0,1]$. As the least squares mean for the Wasserstein metric, the \emph{Wasserstein mean} denoted by $\Omega(\omega; A_{1}, \dots, A_{n})$ is defined by
\begin{equation}
\Omega(\omega; A_{1}, \dots, A_{n}) = \underset{X \in \mathbb{P}_{m}} {\arg \min} \sum^{n}_{j=1} w_{j} d^{2}(X, A_{j}),
\end{equation}
and it coincides with the unique solution $X \in \mathbb{P}_{m}$ of the equation
\begin{equation} \label{E:Wass-eq}
X = \sum_{j=1}^{n} w_{j} (X^{1/2} A_{j} X^{1/2})^{1/2}.
\end{equation}
Note that $\Omega(1-t,t; A, B) = A \diamond_{t} B$, and it has been shown that the Wasserstein mean satisfies the arithmetic-Wasserstein mean inequality. On the other hand, it is shown that the Wasserstein mean does not satisfy the monotonicity and the Wasserstein-harmonic mean inequality: see \cite[Section 5]{BJL}. So it is a natural question whether the Wasserstein mean is the multivariate Lie-Trotter mean.
The main goals of this paper are to provide some properties of the Wasserstein mean and to verify that the Wasserstein mean is the multivariate Lie-Trotter mean by finding a lower bound for Wasserstein mean.

We recall in Section 2 the Wasserstein metric with geodesic and see the Wasserstein distance between $A \diamond_{t} B$ and $A \diamond_{t} C$ for $A, B, C \in \mathbb{P}_{m}$ and $t \in [0,1]$. In Section 3 we introduce the Wasserstein mean of positive definite matrices and explore some interesting properties for the Wasserstein mean such as determinantal inequality. We find bounds for the Wasserstein mean in Section 4, and verify in Section 5 that the Wasserstein mean is the multivariate Lie-Trotter mean. Finally, in Section 6 we show that the equation \eqref{E:Wass-eq} for two positive invertible operators $A$ and $B$ and the positive probability vector $\omega = (1-t, t)$ has the unique solution $X = A \diamond_{t} B$. This naturally gives an open problem to extend the Wasserstein mean of positive definite matrices to positive invertible operators.

\section{Wasserstein metric and geodesics}

Let $\mathbb{M}_{m}$ be the set of all $m \times m$ matrices with complex entries. Let $\mathbb{H}_{m}$ be the real vector space of all $m \times m$ Hermitian matrices, and let $\mathbb{P}_{m} \subset \mathbb{H}_{m}$ be the open convex cone of all positive definite matrices. For $A, B \in \mathbb{H}_{m}$ we denote as $A \leq B$ if and only if $B - A$ is positive semi-definite, and as $A < B$ if and only if $B - A$ is positive definite.

The Frobenius norm $\Vert \cdot \Vert_{2}$ gives rise to the Riemannian structure on the
open convex cone $\mathbb{P}_{m}$ with $\langle X, Y \rangle_{A} = \mathrm{Tr}(A^{-1} X A^{-1} Y)$, where $A \in \mathbb{P}_{m}$ and $X, Y \in T_{A}(\mathbb{P}_{m}) = \mathbb{H}_{m}$. Then $\mathbb{P}_{m}$ is a Cartan-Hadamard Riemannian manifold, a simply
connected complete Riemannian manifold with non-positive sectional curvature (the canonical $2$-tensor is non-negative). The Riemannian trace metric between $A$ and $B$ is given by
\begin{displaymath}
\delta(A, B) = \Vert \log A^{-1/2} B A^{-1/2} \Vert_{2},
\end{displaymath}
and the unique (up to parametrization) geodesic curve on $\mathbb{P}_{m}$ connecting from $A$ to $B$ is
\begin{displaymath}
[0,1] \ni t \mapsto A \#_{t} B := A^{1/2} (A^{-1/2} B A^{-1/2})^{t} A^{1/2},
\end{displaymath}
which is called the \emph{weighted geometric mean} of $A$ and $B$. Note that $A \# B = A \#_{1/2} B$ is the unique midpoint of $A$ and $B$ with respect to the Riemannian trace metric, and is the unique solution $X \in \mathbb{P}_{m}$ of the Riccati equation $X A^{-1} X = B$. See \cite{Bh} for more information.

\begin{lemma} \label{L:Geomean}
Let $A, B, C, D \in \mathbb{P}_{m}$ and let $t \in [0,1]$. Then the following are
satisfied.
\begin{itemize}
\item[(1)] $A \#_{t} B = A^{1-t} B^{t}$ if $A$ and $B$ commute.
\item[(2)] $(a A) \#_{t} (b B) = a^{1-t} b^{t} (A \#_{t} B)$ for any $a, b > 0$.
\item[(3)] $A \#_{t} B = B \#_{1-t} A$.
\item[(4)] $A \#_{t} B \leq C \#_{t} D$ whenever $A \leq C$ and $B \leq D$.
\item[(5)] The map $[0,1] \times \mathbb{P}_{m} \times \mathbb{P}_{m} \to \mathbb{P}_{m}, \ (t, A, B) \mapsto A \#_{t} B$ is continuous.
\item[(6)] $X (A \#_{t} B) X^{*} = (X A X^{*}) \#_{t} (X B X^{*})$ for any nonsingular matrix $X$.
\item[(7)] $(A \#_{t} B)^{-1} = A^{-1} \#_{t} B^{-1}$.
\item[(8)] $[(1 - \lambda) A + \lambda B] \#_{t} [(1 - \lambda) C + \lambda D] \geq (1 - \lambda) (A \#_{t} C) + \lambda (B \#_{t} D)$ for any $\lambda \in [0,1]$.
\item[(9)] $\det (A \#_{t} B) = \det(A)^{1-t} \det(B)^{t}$.
\item[(10)] $[(1-t) A^{-1} + t B^{-1}]^{-1} \leq A \#_{t} B \leq (1-t) A + t B$.
\end{itemize}
\end{lemma}

Bhatia, Jain, and Lim \cite{BJL} have introduced a new metric on $\mathbb{P}_{m}$, called the \emph{Wasserstein metric}, and have established that it gives us the Riemannian metric and the explicit formula of geodesic curve.
For given $A, B \in \mathbb{P}$ the Wasserstein metric $d(A, B)$ is given by
\begin{displaymath}
d(A, B) = \left[ \tr \left( \frac{A + B}{2} \right) - \tr (A^{1/2} B A^{1/2})^{1/2} \right]^{1/2}.
\end{displaymath}
This metric has been of interest in quantum information where it is called the \emph{Bures distance}, and in statistics and the theory of optimal transport where it is called the \emph{Wasserstein metric}. It is the matrix version of the Hellinger distance between probability distributions: for probability vectors $\mathbf{p} = (p_{1}, \dots, p_{m})$ and $\mathbf{q} = (q_{1}, \dots, q_{m})$
\begin{displaymath}
d(\mathbf{p}, \mathbf{q}) = \left[ \frac{1}{2} \sum_{i=1}^{m} (\sqrt{p_{i}} - \sqrt{q_{i}})^{2} \right]^{1/2}.
\end{displaymath}

We see the Wasserstein metric is related to the solution of extremal problem. Let $\mathbb{U}_{m}$ be the compact subset of all $m \times m$ unitary matrices. For given $A \in \mathbb{P}_{m}$ we define the set $\mathcal{F}(A)$ as
\begin{displaymath}
\mathcal{F}(A) = \{ M \in \mathbb{M}_{m}: A = M M^{*} \} = \{ A^{1/2} U: U \in \mathbb{U}_{m} \}.
\end{displaymath}

\begin{theorem} \cite[Theorem 1]{BJL} \label{T:Wasserstein metric}
For any $A, B \in \mathbb{P}_{m}$
\begin{displaymath}
\begin{split}
d(A, B) & = \frac{1}{\sqrt{2}} \ \underset{M \in \mathcal{F}(A), \ N \in \mathcal{F}(B)}{\min} \Vert M - N \Vert_{2} \\
& = \frac{1}{\sqrt{2}} \ \underset{U \in \mathbb{U}_{m}}{\min} \Vert A^{1/2} - B^{1/2} U \Vert_{2}.
\end{split}
\end{displaymath}
The minimum in the second expression is attained at a unitary matrix $U$ occurring in the polar decomposition of $B^{1/2} A^{1/2}$:
\begin{displaymath}
B^{1/2} A^{1/2} = U | B^{1/2} A^{1/2} | = U (A^{1/2} B A^{1/2})^{1/2}.
\end{displaymath}
\end{theorem}

\begin{remark}
We check that $d(A, B)$ is indeed a metric on $\mathbb{P}$ by using Theorem \ref{T:Wasserstein metric}.
\begin{itemize}
\item[(i)] Obviously, $d(A, B) \geq 0$.

\item[(ii)] Assume that $A = B$. Then $U = I$ attains the minimum of $\Vert A^{1/2} - B^{1/2} U \Vert_{2}$ over $U \in \mathbb{U}_{m}$, so the minimum value is $0 = d(A, B)$.
Conversely, assume that $d(A, B) = 0$. Then $\Vert A^{1/2} - B^{1/2} U \Vert_{2} = 0$ when $U = B^{1/2} A^{1/2} (A^{1/2} B A^{1/2})^{-1/2}$. So
\begin{displaymath}
\begin{split}
A^{1/2} = B^{1/2} U & = B^{1/2} B^{1/2} A^{1/2} (A^{1/2} B A^{1/2})^{-1/2} \\
& = B A^{1/2} (A^{-1/2} B^{-1} A^{-1/2})^{1/2} A^{1/2} A^{-1/2} = B (A \# B^{-1}) A^{-1/2}.
\end{split}
\end{displaymath}
Set $X = A \# B^{-1}$. Then $X = B^{-1} A$. By the Riccati equation $X A^{-1} X = B^{-1}$, and so, $B^{-2} A = B^{-1}$. Thus, $A = B$.

\item[(iii)] The Frobenius norm $\Vert \cdot \Vert_{2}$ is unitarily invariant: see \cite[Chapter 5]{HJ}. So
\begin{displaymath}
\Vert A^{1/2} - B^{1/2} U \Vert_{2} = \Vert A^{1/2} U^{*} - B^{1/2} \Vert_{2} = \Vert B^{1/2} - A^{1/2} V \Vert_{2},
\end{displaymath}
where $V = U^{*} \in \mathbb{U}_{m}$. Hence, $d(A, B) = d(B, A)$.

\item[(iv)] Let $A, B, C \in \mathbb{P}_{m}$. By the triangle inequality for Frobenius norm
\begin{displaymath}
\begin{split}
d(A, C) & \leq \frac{1}{\sqrt{2}} \Vert A^{1/2} - C^{1/2} U \Vert_{2} \\
& \leq \frac{1}{\sqrt{2}} ( \Vert A^{1/2} - B^{1/2} V \Vert_{2} + \Vert B^{1/2} V - C^{1/2} U \Vert_{2} ) \\
& = \Vert A^{1/2} - B^{1/2} V \Vert_{2} + \Vert B^{1/2} - C^{1/2} W \Vert_{2},
\end{split}
\end{displaymath}
where $W = U V^{*} \in \mathbb{U}_{m}$. So taking the minimum over all $U, V \in \mathbb{U}_{m}$, we see that $d(A, C) \leq d(A, B) + d(B, C)$.
\end{itemize}
\end{remark}

At this stage we recall a theorem from Riemannian geometry. Let $(\mathcal{M}, g)$ and $(\mathcal{N}, h)$ be Riemannian manifolds with Riemannian metrics $g$ and $h$. A differentiable map $\pi: \mathcal{M} \to \mathcal{N}$ is said to be a \emph{smooth submersion} if its differential $D \pi(p): T_{p} \mathcal{M} \to T_{\pi(p)} \mathcal{N}$ is surjective at every point $p \in \mathcal{M}$.  Let $T_{p} \mathcal{M} = \mathcal{V}_{p} \oplus \mathcal{H}_{p}$ be a decomposition of the tangent space $T_{p} \mathcal{M}$, where $\mathcal{V}_{p} = \ker D \pi(p)$ and $\mathcal{H}_{p} = ( \ker D \pi(p) )^{\bot}$ are called the vertical and horizontal space at $p$. Then $\pi$ is called a \emph{Riemannian submersion} if it is a smooth submersion and the map $D \pi(p): \mathcal{H}_{p} \to T_{\pi(p)} \mathcal{N}$ is isometric for all $p \in \mathcal{M}$.

\begin{theorem} \cite{GHL} \label{T:Riemann}
Let $(\mathcal{M}, g)$ be a Riemannian manifold with Riemannian metrics $g$. Let $G$ be a compact Lie group of isometries of $(\mathcal{M}, g)$ acting freely on $\mathcal{M}$. Let $\mathcal{N} = \mathcal{M} / G$, and let $\pi: \mathcal{M} \to \mathcal{N}$ be the quotient map. Then there exists a unique Riemannian metric $h$ on $\mathcal{N}$ for which $\pi: (\mathcal{M}, g) \to (\mathcal{N}, h)$ is a Riemannian submersion.
\end{theorem}

Note that the general linear group $\mathrm{GL}_{m}$ is a Riemannian manifold with the metric induced by the Frobenius inner product. The group $U_{m}$ of unitary matrices is a compact Lie group of isometries for this metric. The quotient space $\mathrm{GL}_{m} / U_{m}$ is $\mathbb{P}$, and the metric inherited by the quotient space $\mathbb{P}$ is (up to a constant factor)
\begin{displaymath}
\underset{U \in U_{m}}{\min} \Vert A^{1/2} - B^{1/2} U \Vert_{2} = \sqrt{2} d(A, B).
\end{displaymath}
The map $\pi: \mathrm{GL}_{m} \to \mathbb{P}, \ \pi(M) = M M^{*}$ is a smooth submersion, and by Theorem \ref{T:Riemann} there is a unique Riemannian metric on $\mathbb{P}$, which is the Wasserstein metric $d$. From this point of view, the geodesic on $\mathbb{P}_{m}$ joining $A$ and $B$ has been derived in \cite{BJL}. The straight line segment $Z(t) = (1-t) A^{1/2} + t B^{1/2} U$ for $0 \leq t \leq 1$ with a unitary matrix $U = B^{1/2} A^{1/2} (A^{1/2} B A^{1/2})^{-1/2}$ is a geodesic in $\mathrm{GL}_{m}$, and by Theorem \ref{T:Riemann} $\gamma(t) = \pi(Z(t)) = Z(t) Z(t)^{*}$ is a geodesic in $\mathbb{P}_{m}$:
\begin{equation} \label{E:Wass-geodesic}
\begin{split}
\gamma(t) & = (1-t)^{2} A + t^{2} B + t(1-t) [ A (A^{-1} \# B) + (A^{-1} \# B) A ] \\
& = (1-t)^{2} A + t^{2} B + t(1-t) [ (A B)^{1/2} + (B A)^{1/2} ].
\end{split}
\end{equation}
We denote $\gamma(t) =: A \diamond_{t} B$ for $t \in [0,1]$. Since $\gamma(t)$ is the natural parametrization of the geodesic joining $A$ and $B$, it satisfies the affine property of parameters: for any $s, t, u \in [0,1]$
\begin{displaymath}
(A \diamond_{s} B) \diamond_{u} (A \diamond_{t} B) = A \diamond_{(1-u) s + u t} B
\end{displaymath}

\begin{lemma}
For any $A, B, C \in \mathbb{P}_{m}$ and $t \in [0,1]$
\begin{displaymath}
d(A \diamond_{t} B, A \diamond_{t} C) \leq t \sqrt{\frac{\lambda}{2}} \Vert A^{-1} \# B - A^{-1} \# C \Vert_{2},
\end{displaymath}
where $\lambda := \lambda_{1}(A)$ is the largest eigenvalue of $A$.
\end{lemma}

\begin{proof}
Note that $A \diamond_{t} B = Z(t) Z(t)^{*}$, where $Z(t) = (1-t) A^{1/2} + t B^{1/2} U$ for $0 \leq t \leq 1$ with a unitary matrix $U = B^{1/2} A^{1/2} (A^{1/2} B A^{1/2})^{-1/2}$. So $Z(t) = (1-t) A^{1/2} + t (A^{-1} \# B) A^{1/2} \in \mathcal{F}(A \diamond_{t} B)$, since
$$ U = B^{1/2} A^{1/2} (A^{-1/2} B^{-1} A^{-1/2})^{1/2} = B^{1/2} (A \# B^{-1}) A^{-1/2}, $$
and so $B^{1/2} U = B (A \# B^{-1}) A^{-1/2} = (A \# B^{-1})^{-1} A^{1/2} = (A^{-1} \# B) A^{1/2}$ by the Riccati equation and Lemma \ref{L:Geomean} (7). Similarly, $A \diamond_{t} C = Y(t) Y(t)^{*}$, where $Y(t) = (1-t) A^{1/2} + t (A^{-1} \# C) A^{1/2} \in \mathcal{F}(A \diamond_{t} C)$ for $0 \leq t \leq 1$. Therefore, from the first expression in Theorem \ref{T:Wasserstein metric}
\begin{displaymath}
\begin{split}
d(A \diamond_{t} B, A \diamond_{t} C) & \leq \frac{1}{\sqrt{2}} \Vert Z(t) - Y(t) \Vert_{2} \\
& = \frac{t}{\sqrt{2}} \Vert (A^{-1} \# B) A^{1/2} - (A^{-1} \# C) A^{1/2} \Vert_{2} \\
& \leq \frac{t}{\sqrt{2}} \Vert A^{1/2} \Vert_{2} \cdot \Vert A^{-1} \# B - A^{-1} \# C \Vert_{2} \\
& \leq t \sqrt{\frac{\lambda}{2}} \Vert A^{-1} \# B - A^{-1} \# C \Vert_{2}.
\end{split}
\end{displaymath}
The second inequality follows from the sub-multiplicative property of the Frobenius norm: see Section 5.6 in \cite{HJ}, and the last inequality follows from the fact that
\begin{displaymath}
\Vert A^{1/2} \Vert_{2}^{2} = \sum_{i=1}^{m} \lambda_{i}(A) \leq \lambda_{1}(A),
\end{displaymath}
where $\lambda_{1}(A), \dots, \lambda_{m}(A)$ are positive eigenvalues of $A$ in the decreasing order.
\end{proof}

\section{Wasserstein mean}

Let $\mathbb{A} = (A_{1}, \dots, A_{n}) \in \mathbb{P}_{m}^{n}$, and let $\omega = (w_{1}, \dots, w_{n}) \in \Delta_{n}$, the simplex of all positive probability vectors in $\mathbb{R}^{n}$. We consider the following minimization problem
\begin{equation} \label{E:minimization}
\underset{X \in \mathbb{P}_{m}}{\arg \min} \sum_{j=1}^{n} w_{j} d^{2}(X, A_{j}).
\end{equation}
By using tools from non-smooth analysis, convex duality, and the optimal transport theory, it has been proved in Theorem 6.1, \cite{AC} that the above minimization problem has a unique solution in $\mathbb{P}_{m}$. On the other hand, it has been shown in \cite{BJL} that the objective function $\displaystyle f(X) = \sum_{j=1}^{n} w_{j} d^{2}(X, A_{j})$ is strictly convex, by applying the strict concavity of the map $h: \mathbb{P}_{m} \to \mathbb{R}, \ h(X) = \mathrm{Tr} (X^{1/2})$. Therefore, we define such a unique minimizer of \eqref{E:minimization} as the \emph{Wasserstein mean}, denoted by $\Omega(\omega; \mathbb{A})$. That is,
\begin{equation} \label{E:Wasserstein mean}
\Omega(\omega; \mathbb{A}) = \underset{X \in \mathbb{P}}{\arg \min} \sum_{j=1}^{n} w_{j} d^{2}(X, A_{j}).
\end{equation}
% For a uniform probability vector $\omega = (1/n, \dots, 1/n)$
% we simply denote as $\Omega(\omega; \mathbb{A}) = \Omega(\mathbb{A})$.

To find the unique minimum of objective function $f: \mathbb{P}_{m} \to \mathbb{R}$, we evaluate the derivative $D f(X)$ and set it equal to zero.
By using matrix differential calculus, we have the following.
\begin{theorem} \cite[Theorem 8]{BJL} \label{T:WassEq}
The Wasserstein mean $\Omega(\omega; \mathbb{A})$ is a unique solution $X \in \mathbb{P}_{m}$ of the nonlinear matrix equation
\begin{equation} \label{E:Wass}
I = \sum_{j=1}^{n} w_{j} (A_{j} \# X^{-1}),
\end{equation}
equivalently,
\begin{displaymath}
X = \sum_{j=1}^{n} w_{j} (X^{1/2} A_{j} X^{1/2})^{1/2}.
\end{displaymath}
\end{theorem}

We see some interesting properties of the Wasserstein mean. For given $\mathbb{A} = (A_{1}, \dots, A_{n}) \in \mathbb{P}_{m}^{n}$, any permutation $\sigma$ on $\{ 1, \dots, n \}$, and any $M \in \mathrm{GL}_{m}$, we denote as
\begin{displaymath}
\begin{split}
\mathbb{A}_{\sigma} & = (A_{\sigma(1)}, \dots, A_{\sigma(n)}) \in \mathbb{P}_{m}^{n}, \\
M \mathbb{A} M^{*} & = (M A_{1} M^{*}, \dots, M A_{n} M^{*}) \in \mathbb{P}_{m}^{n}, \\
\mathbb{A}^{k} & = (\underline{A_{1}, \dots, A_{n}}, \dots,
\underline{A_{1}, \dots, A_{n}}) \in \mathbb{P}_{m}^{nk},
\end{split}
\end{displaymath}
where the number of blocks in the last expression is $k$. For given $\omega = (w_{1}, \dots, w_{n}) \in \Delta_{n}$, we also denote as
\begin{displaymath}
\begin{split}
\omega_{\sigma} & = (w_{\sigma(1)}, \dots, w_{\sigma(n)}) \in \Delta_{n}, \\
\omega^{k} & = \frac{1}{k} (\underline{w_{1}, \dots, w_{n}}, \dots,
\underline{w_{1}, \dots, w_{n}}) \in \Delta_{nk}.
\end{split}
\end{displaymath}

\begin{proposition}
Let $\mathbb{A} = (A_{1}, \dots, A_{n}) \in \mathbb{P}_{m}^{n}$, and let $\omega = (w_{1}, \dots, w_{n}) \in \Delta_{n}$. Then the following are satisfied.
\begin{itemize}
\item[(1)] $($Homogeneity$)$ \ $\Omega(\omega; \alpha \mathbb{A}) = \alpha \Omega(\omega; \mathbb{A})$ for any $\alpha > 0$.

\item[(2)] $($Permutation invariancy$)$ \ $\Omega(\omega_{\sigma}; \mathbb{A}_{\sigma}) = \Omega(\omega; \mathbb{A})$ for any permutation $\sigma$ on $\{ 1, \dots, n \}$.

\item[(3)] $($Repetition invariancy$)$ \ $\Omega(\omega^{k}; \mathbb{A}^{k}) = \Omega(\omega; \mathbb{A})$ for any $k \in \mathbb{N}$.

\item[(4)] $($Unitary congruence invariancy$)$ \ $\Omega(\omega; U \mathbb{A} U^{*}) = U \Omega(\omega; \mathbb{A}) U^{*}$ for any $U \in \mathbb{U}_{m}$.
\end{itemize}
\end{proposition}

\begin{proof}
Items (2) and (3) follows from the definition \eqref{E:Wasserstein mean} of Wasserstein mean.
\begin{itemize}
\item[(1)] Let $X = \Omega(\omega; \alpha \mathbb{A})$ for any $\alpha > 0$. By Theorem \ref{T:WassEq}
$$ \displaystyle I = \sum_{j=1}^{n} w_{j} (\alpha A_{j}) \# X^{-1} = \sum_{j=1}^{n} w_{j} A_{j} \# (\alpha^{-1} X)^{-1}. $$
By Theorem \ref{T:WassEq} $\alpha^{-1} X = \Omega(\omega; \mathbb{A})$, which implies the desired identity.

\item[(4)] Let $X = \Omega(\omega; U \mathbb{A} U^{*})$ for any $U \in \mathbb{U}_{m}$. By Theorem \ref{T:WassEq} $\displaystyle I = \sum_{j=1}^{n} w_{j} (U A_{j} U^{*} \# X^{-1})$. Taking the congruence transformation by $U^{*} \in \mathbb{U}_{m}$ on both sides and applying Lemma \ref{L:Geomean} (6)
\begin{displaymath}
I = \sum_{j=1}^{n} w_{j} (A_{j} \# U^{*} X^{-1} U) = \sum_{j=1}^{n} w_{j} (A_{j} \# (U^{*} X U)^{-1}).
\end{displaymath}
By Theorem \ref{T:WassEq}, we obtain $U^{*} X U = \Omega(\omega; \mathbb{A})$, that is, $\Omega(\omega; U \mathbb{A} U^{*}) = U \Omega(\omega; \mathbb{A}) U^{*}$.
\end{itemize}
\end{proof}

\begin{remark}
Let
\[
A =
 \left[ {\begin{array}{cc}
  1 & 2 \\
  2 & 5 \\
 \end{array} } \right]
,
\qquad
B =
 \left[ {\begin{array}{cc}
  4 & 4\\
  4 & 5\\
 \end{array} } \right].
\]
One can see easily that $A, B$ are positive definite and $A B \neq B A$. The Wasserstein mean $\Omega \left( \frac{1}{2} , \frac{1}{2} ; A , B \right) = A \diamond B$ and the Riemannian mean $\Lambda \left( \frac{1}{2} , \frac{1}{2} ; A , B \right) = A \# B$ of positive definite matrices $A$ and $B$, respectively, are
\[
   \Omega \left( \frac{1}{2},\frac{1}{2} ; A , B \right)=
 \frac{1}{4} \left[ {\begin{array}{cc}
   9 & 12 \\
   12 & 20 \\
  \end{array} } \right]
,
\qquad
 \Lambda \left( \frac{1}{2} , \frac{1}{2} ; A , B \right)  =
\left[{\begin{array}{cc}
1.6641 & 2.2188\\
2.2188 & 4.1603\\
\end{array} } \right].
\]
Then their determinants are
\begin{displaymath}
\det \left[ \Omega \left( \frac{1}{2} , \frac{1}{2}; A , B \right) \right] = 2.25 > 2 = \det \left[ \Lambda \left( \frac{1}{2} , \frac{1}{2} ; A , B \right) \right].
\end{displaymath}
In general, $\displaystyle \det \Omega(\omega; \mathbb{A}) \neq \prod_{j=1}^{n} (\det A_{j})^{w_{j}} = \det \Lambda(\omega; \mathbb{A})$. The following shows the inequality between determinants of the Wasserstein mean and the Cartan mean.
\end{remark}

It is known from Theorem 7.6.6 in \cite{HJ} that the map $f: \mathbb{P}_{m} \to \mathbb{R}, \ f(A) = \log \det A$ is strictly concave: for any $A, B \in \mathbb{P}_{m}$ and $t \in [0,1]$
\begin{displaymath}
\log \det ( (1-t) A + t B ) \geq (1-t) \log \det A + t \log \det B,
\end{displaymath}
where equality holds if and only if $A = B$. By induction together with this property, we have
\begin{lemma} \label{L:logdet}
Let $A_{1}, \dots, A_{n} \in \mathbb{P}_{m}$, and let $\omega = (w_{1}, \dots, w_{n}) \in \Delta_{n}$. Then
\begin{displaymath}
\log \det \left( \sum_{j=1}^{n} w_{j} A_{j} \right) \geq \sum_{j=1}^{n} w_{j} \log \det A_{j},
\end{displaymath}
where equality holds if and only if $A_{1} = \cdots = A_{n}$.
\end{lemma}

The following shows the determinantal inequality between the Wasserstein mean and the Riemannian mean.

\begin{theorem}
Let $\mathbb{A} = (A_{1}, \dots, A_{n}) \in \mathbb{P}_{m}^{n}$, and let $\omega = (w_{1}, \dots, w_{n}) \in \Delta_{n}$. Then
\begin{equation} \label{E:logdet}
\det \Omega(\omega; \mathbb{A}) \geq \prod_{j=1}^{n} (\det A_{j})^{w_{j}},
\end{equation}
where equality holds if and only if $A_{1} = \cdots = A_{n}$.
\end{theorem}

\begin{proof}
Let $X = \Omega(\omega; \mathbb{A})$. Then by Theorem \ref{T:WassEq} $\displaystyle I = \sum_{j=1}^{n} w_{j} (A_{j} \# X^{-1})$, and by Lemma \ref{L:logdet}
\begin{displaymath}
\begin{split}
0 = \log \det \left[ \sum_{j=1}^{n} w_{j} (A_{j} \# X^{-1}) \right]
& \geq \sum_{j=1}^{n} w_{j} \log \det (A_{j} \# X^{-1}) \\
& = \frac{1}{2} \sum_{j=1}^{n} w_{j} \log \det A_{j} - \frac{1}{2} \log \det X.
\end{split}
\end{displaymath}
The last equality follows from Lemma \ref{L:Geomean} (9). It implies
\begin{displaymath}
\log \det X \geq \sum_{j=1}^{n} w_{j} \log \det A_{j} = \log \left[ \prod_{j=1}^{n} ( \det A_{j} )^{w_{j}} \right].
\end{displaymath}
Taking the exponential function on both sides and applying the fact that the exponential function from $\mathbb{R}$ to $(0, \infty)$ is monotone increasing, we obtain the desired inequality.

Moreover, the equality of \eqref{E:logdet} holds if and only if $A_{i} \# X^{-1} = A_{j} \# X^{-1}$ for all $i$ and $j$. By Lemma \ref{L:Geomean} (3), $X^{-1} \# A_{i} = X^{-1} \# A_{j}$, and by the definition of geometric mean it is equivalent to $A_{i} = A_{j}$ for all $i$ and $j$.
\end{proof}

\section{Bounds for the Wasserstein mean}

The Wasserstein mean satisfies the arithmetic-Wasserstein mean inequality.
\begin{theorem} \cite[Theorem 9]{BJL} \label{T:arith-Wass}
Let $\mathbb{A} = (A_{1}, \dots, A_{n}) \in \mathbb{P}_{m}^{n}$ and let $\omega = (w_{1}, \dots, w_{n}) \in \Delta_{n}$. Then
\begin{displaymath}
\Omega(\omega; \mathbb{A}) \leq \sum_{j=1}^{n} w_{j} A_{j}.
\end{displaymath}
\end{theorem}

\begin{proposition} \label{P:upper bound}
Let $\mathbb{A} = (A_{1}, \dots, A_{n}) \in \mathbb{P}_{m}^{n}$, and let $\omega = (w_{1}, \dots, w_{n}) \in \Delta_{n}$. Then for an operator norm $\Vert \cdot \Vert$
\begin{displaymath}
\Vert \Omega(\omega; \mathbb{A}) \Vert \leq \left( \sum_{j=1}^{n} w_{j} \Vert A_{j} \Vert^{1/2} \right)^{2}.
\end{displaymath}
\end{proposition}

\begin{proof}
Let $X = \Omega(\omega; \mathbb{A})$. Then by Theorem \ref{T:WassEq}, by the triangle inequality for the operator norm, by the fact that $\Vert A^{t} \Vert = \Vert A \Vert^{t}$ for any $A \in \mathbb{P}_{m}$ and $t \geq 0$, and by the sub-multiplicativity for the operator norm in \cite[Section 5.6]{HJ}
\begin{displaymath}
\begin{split}
\Vert X \Vert = \left\Vert \sum_{j=1}^{n} w_{j} \left( X^{1/2} A_{j} X^{1/2} \right)^{1/2} \right\Vert
& \leq \sum_{j=1}^{n} \left\Vert w_{j} \left( X^{1/2} A_{j} X^{1/2} \right)^{1/2} \right\Vert \\
& = \sum_{j=1}^{n} w_{j} \left\Vert  X^{1/2} A_{j} X^{1/2}  \right\Vert^{1/2} \\
& \leq \sum_{j=1}^{n} w_{j} \Vert X \Vert^{1/2} \Vert A_{j} \Vert^{1/2}.
\end{split}
\end{displaymath}
Hence, we obtain
\begin{displaymath}
\Vert \Omega(\omega; \mathbb{A}) \Vert = \Vert X \Vert \leq \left(\sum_{j=1}^{n} w_{j} \Vert A_{j} \Vert^{1/2} \right)^{2}.
\end{displaymath}
\end{proof}

\begin{remark}
By Theorem \ref{T:arith-Wass} we have
\begin{displaymath}
\Vert \Omega(\omega; \mathbb{A}) \Vert \leq \left\Vert \sum_{j=1}^{n} w_{j} A_{j} \right\Vert \leq \sum_{j=1}^{n} w_{j} \Vert A_{j} \Vert.
\end{displaymath}
Since the square map $\mathbb{R} \ni t \mapsto t^{2} \in [0, \infty)$ is convex,
\begin{displaymath}
\left( \sum_{j=1}^{n} w_{j} \Vert A_{j} \Vert^{1/2} \right)^{2} \leq \sum_{j=1}^{n} w_{j} \Vert A_{j} \Vert.
\end{displaymath}
Thus, one can see that Proposition \ref{P:upper bound} gives a sharp upper bound of the Wasserstein mean for the operator norm.
\end{remark}

Unfortunately, the Wasserstein mean does not satisfy the Wasserstein-harmonic mean inequality: see Section 5 in \cite{BJL}. However, we give the lower bound for the Wasserstein mean under certain condition.
\begin{theorem} \label{T:Wass-lower}
Let $\omega = (w_{1}, \dots, w_{n}) \in \Delta_{n}$ and $\mathbb{A} = (A_{1}, \dots, A_{n}) \in \mathbb{P}_{m}^{n}$. Then
\begin{displaymath}
\Omega(\omega; \mathbb{A}) \geq 2 I - \sum_{j=1}^{n} w_{j} A_{j}^{-1}.
\end{displaymath}
\end{theorem}

\begin{proof}
Let $\Omega = \Omega(\omega; \mathbb{A})$. By Theorem \ref{T:WassEq} and the geometric-harmonic mean inequality in Lemma \ref{L:Geomean} (10),
\begin{displaymath}
I = \sum_{j=1}^{n} w_{j} (A_{j} \# \Omega^{-1}) \geq \sum_{j=1}^{n} w_{j} \left( \frac{A_{j}^{-1} + \Omega}{2} \right)^{-1}.
\end{displaymath}
Taking inverse on both sides and applying the convexity of inversion map in (1.33) of \cite{Bh} yield
\begin{displaymath}
% \begin{split}
I \leq \left[ \sum_{j=1}^{n} w_{j} \left( \frac{A_{j}^{-1} + \Omega}{2} \right)^{-1} \right]^{-1}
\leq \sum_{j=1}^{n} w_{j} \left( \frac{A_{j}^{-1} + \Omega}{2} \right) = \frac{1}{2} \sum_{j=1}^{n} w_{j} A_{j}^{-1} + \frac{1}{2} X.
% \end{split}
\end{displaymath}
By a simple calculation, we obtain the desired inequality.
\end{proof}

\begin{remark}
Note that $\displaystyle 2 I - \sum_{j=1}^{n} w_{j} A_{j}^{-1} \leq \left[ \sum_{j=1}^{n} w_{j} A_{j}^{-1} \right]^{-1}$. Indeed,
\begin{displaymath}
I = \left( \sum_{j=1}^{n} w_{j} A_{j}^{-1} \right) \# \left( \sum_{j=1}^{n} w_{j} A_{j}^{-1} \right)^{-1} \leq \frac{1}{2} \left[ \sum_{j=1}^{n} w_{j} A_{j}^{-1} + \left( \sum_{j=1}^{n} w_{j} A_{j}^{-1} \right)^{-1} \right].
\end{displaymath}
\end{remark}

We give another upper bound for the Wasserstein mean different from the arithmetic mean.

\begin{remark}
Assume that $\displaystyle \sum_{j=1}^{n} w_{j} A_{j} < 2 I$. Let $\Omega = \Omega(\omega; \mathbb{A})$. By Theorem \ref{T:WassEq} and the arithmetic-geometric mean inequality in Lemma \ref{L:Geomean} (10),
\begin{displaymath}
I = \sum_{j=1}^{n} w_{j} (A_{j} \# \Omega^{-1}) \leq \sum_{j=1}^{n} w_{j} \left( \frac{A_{j} + \Omega^{-1}}{2} \right).
\end{displaymath}
By a simple calculation, we have $\displaystyle 0< 2 I - \sum_{j=1}^{n} w_{j} A_{j} \leq \Omega^{-1}$, and so
\begin{displaymath}
\Omega \leq \left[ 2 I - \sum_{j=1}^{n} w_{j} A_{j} \right]^{-1}.
\end{displaymath}
This means that $\displaystyle \left[ 2 I - \sum_{j=1}^{n} w_{j} A_{j} \right]^{-1}$ is an upper bound for $\Omega(\omega; A_{1}, \dots, A_{n})$.

On the other hand, note that $\displaystyle \left[ 2 I - \sum_{j=1}^{n} w_{j} A_{j} \right]^{-1} \geq \sum_{j=1}^{n} w_{j} A_{j}$. Indeed,
\begin{displaymath}
I = \left( \sum_{j=1}^{n} w_{j} A_{j} \right) \# \left( \sum_{j=1}^{n} w_{j} A_{j} \right)^{-1} \leq \frac{1}{2} \left[ \sum_{j=1}^{n} w_{j} A_{j} + \left( \sum_{j=1}^{n} w_{j} A_{j} \right)^{-1} \right].
\end{displaymath}
Then $\displaystyle 2 I - \sum_{j=1}^{n} w_{j} A_{j} \leq \left( \sum_{j=1}^{n} w_{j} A_{j} \right)^{-1}$, so $\displaystyle \left[ 2 I - \sum_{j=1}^{n} w_{j} A_{j} \right]^{-1} \geq \sum_{j=1}^{n} w_{j} A_{j}$.
\end{remark}

\section{Applications to the Lie-Trotter mean}

We see in this section some applications of the lower bound of the Wasserstein mean in Theorem \ref{T:Wass-lower} to the notion of Lie-Trotter means. A weighted $n$-mean $G_{n}$ on $\mathbb{P}_{m}$ for $n \geq 2$ is a map $G_{n}(\omega; \cdot): \mathbb{P}_{m}^{n} \to \mathbb{P}_{m}$ that is idempotent, in the sense that $G_{n}(\omega; X, \ldots, X) = X$ for all $X \in \mathbb{P}_{m}$. A weighted $n$-mean $G_{n}(\omega; \cdot): \mathbb{P}_{m}^{n} \to \mathbb{P}_{m}$ is called a \emph{multivariable Lie-Trotter mean} if it is differentiable and satisfies
\begin{equation} \label{E:Lie-Trotter}
\displaystyle \lim_{s \to 0} G_{n} (\omega; \gamma_{1}(s), \gamma_{2}(s), \ldots, \gamma_{n}(s))^{1/s} = \exp \left[ \sum_{i=1}^{n} w_{i} \gamma_{i}'(0) \right],
\end{equation}
where for $\epsilon > 0$, $\gamma_{i} : (-\epsilon, \epsilon) \to \mathbb{P}_{m}$ are differentiable curves with $\gamma_{i}(0) = I$ for all $i = 1, \dots, n$. See \cite{HK} for more details and information.

\begin{lemma}
Let $\Omega^{\omega} := \Omega(\omega; \cdot): \mathbb{P}_{m}^{n} \to \mathbb{P}_{m}$ be the Wasserstein mean for given probability vector $\omega = (w_{1}, \dots, w_{n})$. Then it is differentiable at $\mathbb{I} = (I, \dots, I)$ with
\begin{displaymath}
D \Omega^{\omega}(\mathbb{I})(X_{1}, \dots, X_{n}) = \sum_{j=1}^{n} w_{j} X_{j}.
\end{displaymath}
\end{lemma}

\begin{proof}
Let $X_{1},\dots,X_{n} \in S(\mathbb{H}).$
If $X_{1}=\dots=X_{n}=0,$ then the statement holds obviously. Without loss of generality, we assume that at least one of $X_{1},\dots,X_{n}$ is not zero. Set
\begin{displaymath}
\rho:=\max_{1\leq j \leq n}\sigma(X_{j})
\end{displaymath}
where $\sigma(X)$ is the spectral radius of $X.$ Then $\rho > 0.$ Define
\begin{displaymath}
f(t)=2I-\sum_{1}^{n}\omega_{j}(i+tX_{j})^{-1}
\end{displaymath}
on $(-\frac{1}{\rho},\frac{1}{\rho}).$ Then
\begin{eqnarray*}
\lambda(I+tX_{j})=1+t\lambda(X_{j})  \geq  1-|t||\lambda(X_{j})|
 \geq  1-\rho|t| >  0
\end{eqnarray*}
where $\lambda(X)$ denote the eigenvalue of $X.$ So $I+tX_{j} \in \mathbb{P}$ for any $t \in (-\frac{1}{\rho},\frac{1}{\rho}).$ Thus $f$ is well-defined in a neighborhood $(-\frac{1}{\rho},\frac{1}{\rho})$ of $0$ and $f(0)=2I-\sum_{j=1}^{n}\omega_{j}(I)^{-1}=I.$ Since the derivative of the map$t \mapsto (tX+I)^{-1}$ at $t=0$ is $-X.$ We have
\begin{eqnarray*}
f'(0) & = & \lim_{t \rightarrow 0}\frac{I-\sum_{j=1}^{n}\omega_{j}(I+tX_{j})^{-1}}{t}\\
& = & \lim_{t \rightarrow 0}\sum_{j=1}^{n}\omega_{j}(I+tX_{j})^{-1}X_{j}(I+tX_{j})^{-1} =  \sum_{j=1}^{n}\omega_{j}X_{j}.
\end{eqnarray*}
Then by Theorem \ref{T:arith-Wass} and \ref{T:Wass-lower},
\begin{eqnarray*}
\frac{[2I-\sum_{j=1}^{n}\omega_{j}(I+tX_{j})^{-1}]-I}{t} & \leq & \frac{\Omega_{n}^{\omega}(\omega ; I+tX_{j})-I}{t} \\
& \leq & \frac{\sum_{j=1}^{n}\omega_{j}(I+tX_{j})-I}{t} =  \sum_{j=1}^{n}\omega_{j}X_{j}
\end{eqnarray*}
for any sufficiently small $t > 0.$
So we have
\begin{displaymath}
\lim_{t \rightarrow 0^{+}}\frac{\Omega_{n}^{w}(\omega; I+tX_{1},\dots,I+tX_{n})-\Omega_{n}^{w}(I,\dots,I)}{t}=\sum_{j=1}^{n}\omega_{j}X_{j}.
\end{displaymath}
Since $\Omega_{n}^{\omega}(I,\dots,I)=I.$ Similarly, for $t < 0$
\begin{displaymath}
\lim_{t \rightarrow 0^{-}}\frac{\Omega_{n}^{\omega}(\omega;I+tX_{1},\dots,I+tX_{n})-\Omega_{n}^{\omega}(I,\dots,I)}{t}=\sum^{n}_{j=1}\omega_{j}X_{j}.
\end{displaymath}
We conclude that $\Omega_{n}^{w}$ is differentiable at $\mathbb{I}$ with $D\Omega_{n}^{\omega}(\mathbb{I})(X_{1},\dots,X_{n})=\sum^{n}_{j=1}\omega_{j}X_{j}.$
\end{proof}

\begin{theorem} \label{T:LT-Wass}
The Wasserstein mean is the multivariate Lie-Trotter mean, that is, for given $\omega = (w_{1}, \dots, w_{n}) \in \Delta_{n}$
\begin{displaymath}
\displaystyle \lim_{s \to 0} \Omega (\omega; \gamma_{1}(s), \gamma_{2}(s), \ldots, \gamma_{n}(s))^{1/s} = \exp \left[ \sum_{j=1}^{n} w_{j} \gamma_{j}'(0) \right],
\end{displaymath}
where for $\epsilon > 0$, $\gamma_{j} : (-\epsilon, \epsilon) \to \mathbb{P}_{m}$ are differentiable curves with $\gamma_{j}(0) = I$ for all $j = 1, \dots, n$.
\end{theorem}

\begin{proof}
Let $\omega=(\omega_{1},\dots,\omega_{n-1},\omega_{n}) \in \Delta_{n}$ and let $\gamma_{1},\dots,\gamma_{n} : (-\epsilon,\epsilon)\mapsto \mathbb{P}$ be any differentiable curve with $\gamma_{j}(0)=I$ for all $i=1,\dots,n.$ Then
\begin{displaymath}
2I-\sum_{j=1}^{n}\omega_{j}\gamma_{j}(s)^{-1} \leq \Omega(\omega;\gamma_{1}(s),\dots,\gamma_{n}(s)) \leq \sum_{j=1}^{n}\omega_{j}\gamma_{j}(s).
\end{displaymath}
Taking logarithms and using the fact that the logarithm function is operator monotone, we have
\begin{displaymath}
\log\left(2I-\sum_{j=1}^{n}\omega_{j}\gamma_{j}(s)^{-1}\right) \leq \log\Omega(\omega;\gamma_{1}(s),\dots,\gamma_{n}(s)) \leq \log\sum_{j=1}^{n}\omega_{j}\gamma_{j}(s).
\end{displaymath}
For $s>0,$ multiplying all terms by $1/s,$ we get
\begin{displaymath}
\frac{1}{s}\log\left(2I-\sum_{j=1}^{n}\gamma_{j}(S)^{-1}\right) \leq \log\Omega(\omega;\gamma_{1}(s),\dots,\gamma_{n}(s))^{1/s} \leq \frac{1}{s}\log\sum_{j=1}^{n}\omega_{j}\gamma_{j}(s).
\end{displaymath}
Taking the limit $s \rightarrow 0^{+},$ and using the l'H$\hat{o}$pital's theorem we obtain
\begin{eqnarray*}
\lim_{s \rightarrow 0^{+}}\log\Omega(\omega;\gamma_{1}(s),\dots,\gamma_{n}(s))^{1/s}
 = \sum_{j=1}^{n}\omega_{j}\gamma_{j}'(0).
\end{eqnarray*}
Since the logarithm map $\log:\mathbb{P} \rightarrow S(\mathbb{H})$ is diffeomorphic,
\begin{displaymath}
\lim_{s \rightarrow 0^{+}}\Omega(\omega;\gamma_{1}(s),\dots,\gamma_{n}(s))^{1/s}
 = \exp\left[\sum_{j=1}^{n}\omega_{j}\gamma_{j}'(0)\right].
\end{displaymath}
For $s < 0,$ we obtain $ \displaystyle \lim_{s \rightarrow 0^{-}}(\omega;\gamma_{1}(s),\dots,\gamma_{n}(s))^{1/s}=\exp\left[\sum_{j=1}^{n}\omega_{j}\gamma_{j}'(0)\right]$ by similar steps.
\end{proof}

Taking $\gamma_{i}(s) = A_{i}^{s}$ for each $i$ and some $A_{i} \in \mathbb{P}_{m}$, we obtain from Theorem \ref{T:LT-Wass}
\begin{corollary}
Let $A_{1}, \dots, A_{n} \in \mathbb{P}_{m}$ and $\omega = (w_{1}, \dots, w_{n}) \in \Delta_{n}$. Then
\begin{displaymath}
\displaystyle \lim_{s \to 0} \Omega (\omega; A_{1}^{s}, \ldots, A_{n}^{s})^{1/s} = \exp \left[ \sum_{i=1}^{n} w_{i} \log A_{i} \right].
\end{displaymath}
\end{corollary}

\section{Final remarks}

It is a natural question if the Wasserstein mean can be defined on the setting $\mathbb{P}$ of positive definite operators. Since one can not have the Wasserstein metric on $\mathbb{P}$, the definition \eqref{E:Wasserstein mean} may not be available. One possible approach to define the operator Wasserstein mean is to show the existence and uniqueness of the solution of the equation \eqref{E:Wass}.
On the other hand, one can not find the explicit form of the solution of the nonlinear equation \eqref{E:Wass}, but we have seen that the solution of \eqref{E:Wass} for two positive definite matrices $A$ and $B$ coincides with the geodesic $\gamma(t) = A \diamond_{t} B$ in \eqref{E:Wass-geodesic} with respect to the Wasserstein metric. We directly solve the nonlinear equation \eqref{E:Wass} for $n = 2$ by using the properties of geometric mean of positive definite operators.

For positive definite operators $A, B \in \mathbb{P}$ and $t \in [0,1]$ the weighted geometric mean of $A$ and $B$ is defined by
\begin{displaymath}
A \#_{t} B = A^{1/2} (A^{-1/2} B A^{-1/2})^{t} A^{1/2}.
\end{displaymath}
Note that $A \# B = A \#_{1/2} B$ is the unique positive definite solution $X \in \mathbb{P}$ of the Riccati equation $X A^{-1} X = B$. Moreover, it satisfies most of all properties in Lemma \ref{L:Geomean}, but we list some of them that are useful for our goal. See \cite{CPR,KA,LL}.
\begin{lemma} \label{L:O-Geomean}
Let $A, B, C, D \in \mathbb{P}$ and let $t \in [0,1]$. Then the following are satisfied.
\begin{itemize}
% \item[(1)] $A \#_{t} B = A^{1-t} B^{t}$ if $A$ and $B$ commute.
% \item[(2)] $(a A) \#_{t} (b B) = a^{1-t} b^{t} (A \#_{t} B)$ for any $a, b > 0$.
\item[(1)] $A \#_{t} B = B \#_{1-t} A$.
% \item[(2)] $A \#_{t} B \leq C \#_{t} D$ whenever $A \leq C$ and $B \leq D$.
% \item[(5)] The map $[0,1] \times \mathbb{P}_{m} \times \mathbb{P}_{m} \to \mathbb{P}_{m}, \ (t, A, B) \mapsto A \#_{t} B$ is continuous.
\item[(2)] $X (A \#_{t} B) X^{*} = (X A X^{*}) \#_{t} (X B X^{*})$ for any nonsingular matrix $X$.
\item[(3)] $(A \#_{t} B)^{-1} = A^{-1} \#_{t} B^{-1}$.
% \item[(8)] $[(1 - \lambda) A + \lambda B] \#_{t} [(1 - \lambda) C + \lambda D] \geq (1 - \lambda) (A \#_{t} C) + \lambda (B \#_{t} D)$ for any $\lambda \in [0,1]$.
% \item[(9)] $\det (A \#_{t} B) = \det(A)^{1-t} \det(B)^{t}$.
% \item[(10)] $[(1-t) A^{-1} + t B^{-1}]^{-1} \leq A \#_{t} B \leq (1-t) A + t B$.
\end{itemize}
\end{lemma}

\begin{theorem} \label{T:2-Wass}
Let $A, B \in \mathbb{P}$ and $t \in [0,1]$. Then the nonlinear equation
\begin{equation} \label{E:2-Wass}
I = (1-t) (A \# X^{-1}) + t (B \# X^{-1})
\end{equation}
has a unique positive definite solution $X = A \diamond_{t} B$.
\end{theorem}

\begin{proof}
Pre- and post-multiplying all terms by $A^{-1/2}$ for $A > 0$ and by Lemma \ref{L:O-Geomean}, we get
\begin{displaymath}
A^{-1} = (1-t)(A^{-1/2}X^{-1}A^{-1/2})^{1/2} + t(A^{-1/2}BA^{-1/2} \# A^{-1/2}X^{-1}A^{-1/2}).
\end{displaymath}
Let $Y = A^{-1/2}X^{-1}A^{-1/2}$ and $Z = A^{-1/2}BA^{-1/2}.$ Then we have \begin{displaymath}
A^{-1} = (1-t)Y^{1/2} + t(Z \# Y).
\end{displaymath}
By using the Riccati equation, we get
\begin{displaymath} \frac{1}{t^{2}}[A^{-1}-(1-t)Y^{1/2}]Y^{-1}[A^{-1}-(1-t)Y^{1/2}] = Z.
\end{displaymath}
Pre- and post-multiplying all terms by $A$ for $A > 0,$ we get
\begin{displaymath}
\left[ Y^{-1/2} - (1-t) A \right]^{2} = t^{2} AZA.
\end{displaymath}
Taking square root on both sides, we obtain $Y^{-1/2} = (1-t)A + t(AZA)^{1/2}.$ By assumption, we have
\begin{displaymath}
(A^{1/2}XA^{1/2})^{1/2}= (1-t)A + t(A^{1/2}BA^{1/2})^{1/2}.
\end{displaymath}
Taking square on both sides, pre- and post-multiplying all terms by $A^{-1/2}$ for $A > 0,$ we obtain
\begin{displaymath}
X = A^{-1/2} [ (1-t)A + t(A^{1/2}BA^{1/2})^{1/2} ]^{2} A^{-1/2} = A \diamond_{t} B.
\end{displaymath}
\end{proof}

\textbf{Open question}. For positive definite operators $A_{1}, \dots, A_{n}$ and a positive probability vector $(w_{1}, \dots, w_{n})$, the nonlinear equation
\begin{displaymath}
I = \sum_{j=1}^{n} w_{j} (A_{j} \# X^{-1}),
\end{displaymath}
has a unique positive definite solution $X$ in the setting $\mathbb{P}$ of positive definite operators?

This is an interesting and challangeable problem, and Theorem \ref{T:2-Wass} gives us a positive answer.

\vspace{4mm}

\textbf{Acknowledgement} \\

The work of S. Kim was supported by the National Research Foundation of Korea (NRF) grant funded by the Korea government (MIST) (No. NRF-2018R1C1B6001394).

\end{document}